\documentclass[11pt]{article}
\usepackage{a4wide}
\usepackage{amsmath,amsfonts,amssymb,amsthm,amscd}

\usepackage{hyperref}

\usepackage{epsfig}

\newtheorem{theorem}{Theorem}                            
\newtheorem{proposition}{Proposition}            
\newtheorem{lemma}[theorem]{Lemma}

\newtheorem{problem}{Problem}

\newtheorem*{theorem_}{Theorem}{\bf}{\it}
\newtheorem*{theorem_A}{Theorem A}{\bf}{\it}
{\bf}{\it}
{\bf}{\it}
\newtheorem*{definition_}{Definition}{\bf}{\rm}
{\bf}{\rm}

\newcommand{\cro}{{\mbox {\sc cr}}}

\def\inst#1{$^{#1}$}

\begin{document}


\title{Improvement on the decay of crossing numbers~\thanks{The first and the second author were supported by
Program "Structuring the European Research Area",
Grant MRTN-CT-2004-511953 at R\'enyi Institute and by project 
CE-ITI (GACR P2020/12/G061) of the Czech Science Foundation.
The second author was also partially supported by GraDR EUROGIGA project GIG/11/E023 and by the grant SVV-2010-261313 (Discrete Methods and Algorithms).
The third author was supported by the Hungarian Research Fund grants OTKA T-038397, T-046246, K-60427, and K-83767. 
A preliminary version of this paper appeared in the proceedings of Graph Drawing 2007~\cite{CKT08}. 
}
}

\author{Jakub \v Cern\'y\inst{1}, 
        Jan Kyn\v{c}l\inst{2}, 
        G\'eza T\'oth\inst{3}
} 


\maketitle

\begin{center}
{\footnotesize
\inst{1} 
Department of Applied Mathematics, \\
Charles University, Faculty of Mathematics and Physics, \\
Malostransk\'e n\'am.~25, 118~ 00 Praha 1, Czech Republic; \\ 
\texttt{kuba@kam.mff.cuni.cz}
\\\ \\
\inst{2} 
Department of Applied Mathematics and Institute for Theoretical Computer Science, \\
Charles University, Faculty of Mathematics and Physics, \\
Malostransk\'e n\'am.~25, 118~ 00 Praha 1, Czech Republic; \\ 
\texttt{kyncl@kam.mff.cuni.cz}
\\\ \\
\inst{3}
R\'enyi Institute, Hungarian Academy of Sciences, 
Budapest; \\
\texttt{geza@renyi.hu}
}
\end{center}  

\begin{abstract}
We prove that the crossing number of a graph decays in a ``continuous
fashion''
in the following sense. For any $\varepsilon>0$ there is a $\delta>0$
such that for a sufficiently large $n$, every graph $G$ with $n$ vertices
and $m\ge n^{1+\varepsilon}$ edges,
has a subgraph $G'$ of at most $(1-\delta)m$ edges
and crossing number at least $(1-\varepsilon)\cro(G)$.
This generalizes the result of J. Fox and Cs. T\'oth.
\end{abstract}


\section{Introduction}

For any graph $G$, let $n(G)$ (resp. $m(G)$)
denote the number of its vertices (resp. edges).
If it is clear from the context, we simply write $n$ and $m$
instead of $n(G)$ and $m(G)$.
The crossing number $\cro(G)$ of a graph $G$ is the minimum number of
edge crossings over all drawings of $G$ in the plane.
In the optimal drawing of $G$, crossings are not necessarily distributed
uniformly
among the edges. Some edges could be more ``responsible''
for the crossing number than some other edges.
For any fixed $k$, it is not hard to construct a graph $G$ whose crossing
number is $k$, but $G$ has an edge $e$ such that $G\setminus e$ is planar.
Richter and Thomassen \cite{RT93} started to investigate the following
general problem.
We have a graph $G$, and we want to remove a given number of edges.
By {\em at least} how much does the crossing number decrease?
They conjectured that there is a constant $c$ such that every graph $G$
with $\cro(G)=k$
has an edge $e$ with $\cro(G - e)\ge k-c\sqrt{k}$.
They only proved that $G$ has an edge with
$\cro(G-e)\ge {2\over 5}k-O(1)$.

Pach, Radoi\v ci\'c, Tardos, and T\'oth \cite{PRTT06} proved that
for every graph $G$ with $m(G)\ge {103\over 16}n(G)$, we have
$\cro(G)\ge 0.032{m^3\over n^2}$.
It is not hard to see \cite{PT00} that for {\em any} edge $e$,
we have 
$\cro(G-e)\ge \cro(G)-m+1$. These two results imply an improvement
of the Richter--Thomassen bound if $m\ge 8.1n$, and also imply
the Richter--Thomassen
conjecture for graphs of $\Omega(n^2)$ edges.

J. Fox and Cs. T\'oth \cite{FT08}
investigated the case where we want to delete a {\em
positive fraction} of the edges.

\medskip

\begin{theorem_A}[\cite{FT08}]
For every $\varepsilon>0$, there is an $n_{\varepsilon}$ such that
every graph $G$ with $n(G)\ge n_{\varepsilon}$ vertices and
$m(G)\ge n(G)^{1+\varepsilon}$ edges has a subgraph $G'$
with
$$m(G')\le \left(1-{\varepsilon\over 24}\right)m(G)$$
and
$$\cro(G')\ge \left({1\over 28}-o(1)\right)\cro(G).$$
\end{theorem_A}

In this note we generalize Theorem A.

\begin{theorem_}
For every $\varepsilon, \gamma > 0$, there is an $n_{\varepsilon,\gamma}$
such that every graph $G$ with $n(G)\ge n_{\varepsilon,\gamma}$ vertices
and $m(G)\ge n(G)^{1+\varepsilon}$ edges has a subgraph $G'$ with 
$$m(G')\le \left(1-\frac{\varepsilon\gamma}{1224}\right)m(G)$$
and 
$$\cro(G')\ge (1-\gamma) \cro(G).$$
\end{theorem_}


\section{Proof of the Theorem}

Our proof is based on the argument of Fox and T\'oth \cite{FT08}, 
the only new ingredient is Lemma~\ref{lemma_rcycles}. 

\smallskip

\begin{definition_}
Let $r\ge 2, p\ge 1$ be integers.
A {\it $2r$-earring of size $p$} is a graph which is a union 
of an edge $uv$ and $p$ 
edge-disjoint
paths between $u$ and $v$, each of length at most $2r-1$. 
Edge $uv$ is called the
{\it main edge} of the $2r$-earring.
\end{definition_}


\begin{lemma}
\label{lemma_rcycles}
Let $r\ge 2, p\ge 1$ be integers.
There exists $n_0$ such that
every graph $G$ with $n\ge n_0$ vertices and $m \ge 6pr n^{1+{1/r}}$ edges
contains at least $m/3pr$ edge-disjoint $2r$-earrings, each of size $p$.
\end{lemma}

\begin{proof}
By the result of 
Alon, Hoory, and Linial \cite{AHL02},
for some $n_0$,
every graph with $n\geq n_0$ vertices and at least $n^{1+{1/r}}$ edges
contains a cycle of length at most $2r$. 

Suppose that $G$ has $n\ge n_0$ vertices and $m \ge 6pr n^{1+{1/r}}$
edges.
Take a {\em maximal} edge-disjoint set $\{E_1, E_2, \ldots , E_x\}$
of $2r$-earrings, each of size $p$. Let $E=E_1\cup E_2\cup \cdots \cup
E_x$,
the set of all edges of the earrings and let $G'=G - E$.
Now let $E'_1$ be a $2r$-earring of $G'$ of maximum size.
Note that this size is less than $p$.
Let $G'_1=G'- E'_1$. Similarly, let $E'_2$ be a $2r$-earring of
$G'_1$
of maximum size
and let $G'_2=G'_1- E'_2$. Continue analogously, as long as there
is
a $2r$-earring in the remaining graph. We obtain 
the $2r$-earrings $E'_1, E'_2, \ldots , E'_y$, and the remaining graph 
$G''=G'_y$ does not contain any $2r$-earring. Let
$E'=E'_1\cup E'_2\cup \cdots \cup E'_y$.

We claim that $y< n^{1+{1/r}}$. Suppose on the contrary that 
$y\ge n^{1+{1/r}}$. Take the main edges of $E'_1, E'_2, \ldots , E'_y$.
We have at least 
$n^{1+{1/r}}$ edges so by the result of 
Alon, Hoory, and Linial \cite{AHL02} some of them form a cycle $C$ of
length
at most $2r$. Let $i$ be the smallest index with the property that $C$
contains the main edge of $E'_i$. Then $C$, together with $E'_i$ 
would be a $2r$-earring of $G'_{i-1}$ of greater size than $E'_i$,
contradicting the maximality of $E'_i$.

Each of the earrings $E'_1, E'_2, \ldots , E'_y$
has at most $(p-1)(2r-1)+1$ edges so we have $|E'|\le y(p-1)(2r-1)+y<
(2pr-1)n^{1+{1/r}}$. 
The remaining graph, $G''$ does not contain any $2r$-earring, in
particular, 
it does not contain any cycle of length at most $2r$, since it is 
a $2r$-earring of size one. Therefore, by \cite{AHL02},
for the number of its edges we have $e(G'')< n^{1+{1/r}}$.

It follows that the set $E=\{E_1, E_2, \ldots , E_x\}$
contains at least $m-2prn^{1+{1/r}}$ $\ge {2\over 3}m$ edges.
Each of $E_1, E_2, \ldots , E_x$ has at most $p(2r-1)+1\le 2pr$ edges,
therefore, $x\ge m/3pr$. 
\end{proof}


\begin{lemma}[\cite{FT08}]
\label{lemma_FT}
Let $G$ be a graph with $n$ vertices, 
$m$ edges, 
and degree sequence
$d_1\le d_2\le\cdots\le d_n$.
Let $\ell$ be the integer such that
$\sum_{i=1}^{\ell-1}d_i<4m/3$ but
$\sum_{i=1}^{\ell}d_i\ge 4m/3$.
If $n$ is large enough and $m=\Omega(n\log^2n)$ then 
$$\cro(G)\ge{1\over 65}\sum_{i=1}^{\ell}d^2_i.$$
\end{lemma}

\medskip

\noindent {\it Proof of the Theorem.}
Let $\varepsilon,\gamma \in (0,1)$ be fixed. 
Choose integers $r,p$ such that 
$\frac{1}{\varepsilon}<r\le \frac{2}{\varepsilon}$, and
$\frac{67}{\gamma}<p\le \frac{68}{\gamma}$.
It follows that we have
$\frac{1}{r}<\varepsilon \le \frac{2}{r}$, 
and $\frac{67}{p}<\gamma \le \frac{68}{p}$.
Then there is an $n_{\varepsilon,\gamma}$ with the following properties: 
(a) $n_{\varepsilon,\gamma}\ge n_0$ from Lemma~\ref{lemma_rcycles},
(b) $(n_{\varepsilon,\gamma})^{1+\varepsilon} > 18pr \cdot
(n_{\varepsilon,\gamma})^{1+1/r}$.

Let $G$ be a graph with $n\ge n_{\varepsilon,\gamma}$ vertices and $m\ge
n^{1+\varepsilon}$ edges.
Let $v_1, \ldots , v_n$ be the vertices of $G$, of degrees
$d_1\le d_2\le\cdots\le d_n$ and define 
$\ell$ as in Lemma~\ref{lemma_FT}, that is,
$\sum_{i=1}^{\ell-1}d_i<4m/3$ but
$\sum_{i=1}^{\ell}d_i\ge 4m/3$.
Let $G_0$ be the subgraph of $G$ induced by $v_1, \ldots , v_{\ell}$.
Observe that $G_0$ has $m'\ge m/3$ edges. 
Therefore, by Lemma~\ref{lemma_rcycles} $G_0$ contains at least 
$m'/3pr\ge m/9pr$ edge-disjoint $2r$-earrings, each of size $p$.

Let $M$ be the set of 
the main edges of these $2r$-earrings. We have $|M|\ge m/9pr \ge
\frac{\varepsilon\gamma}{1224}m$.
Let $G'=G- M$ and $G_0'=G_0- M$.

Take an optimal drawing $D(G')$ of the subgraph
$G'\subset G$. We have to draw the missing edges to obtain a drawing of
$G$. 
Our method is a randomized variation of the embedding method, which has
been applied by
Leighton \cite{L83}, Richter and
Thomassen \cite{RT93}, Shahrokhi et al. \cite{SSSV97}, Sz\'ekely
\cite{S04}, 
and most recently by Fox and
T\'oth \cite{FT08}. 
For every missing edge $e_i=u_iv_i\in M\subset G_0$,
$e_i$ is the deleted main 
edge of a 
$2r$-earring $E_i\subset G_0$. So there are $p$ 
edge-disjoint paths in $G_0$ from $u_i$ to $v_i$.
For each of these paths, draw a curve from $u_i$ to $v_i$ 
infinitesimally close to that path, on either side. Call these $p$ curves 
{\em potential $u_iv_i$-edges} and 
call the resulting drawing $D$.
Note that a potential $u_iv_i$-edge crosses itself if the corresponding 
path does. In such cases, we redraw the potential $u_iv_i$-edge in the neighborhood of each self-crossing to get a noncrossing curve.

To get a drawing of $G$, for each $e_i=u_iv_i\in M$,
choose one of the $p$ potential $u_iv_i$-edges
at random, independently and uniformly, with probability
$1/p$, and draw the edge $u_iv_i$ as that curve. 

There are two types of new crossings in the obtained drawing of $G$.
First category crossings are
infinitesimally close to a crossing in $D(G')$, second category crossings
are
infinitesimally close to a vertex of $G_0$ in $D(G')$.

The expected number of first category crossings is at most
$$\left(1+{2\over p}+{1\over p^2}\right)\cro(G')=\left(1+{1\over
p}\right)^2\cro(G').$$
Indeed, for each edge of $G'$, there can be at most one new edge 
drawn next to it, and that is drawn with probability at most $1/p$.
Therefore, in the close neighborhood of a crossing 
in $D(G')$, the expected number of crossings is 
at most $(1+{2\over p}+{1\over p^2})$. See figure~\ref{fig:embedding}(a).

\begin{figure}[ht]
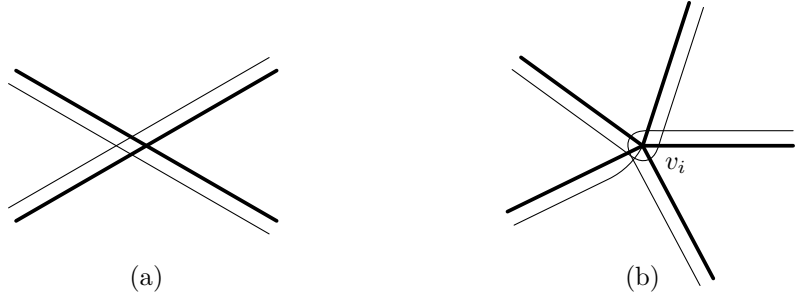

$$\epsfbox{figures.1}\hskip 3cm\epsfbox{figures.2}$$
\caption{The thick edges are edges of $G'$, the thin edges are the
potential
edges. Figure shows
(a) a neighborhood of a crossing in $D(G')$ and (b) a neighborhood of a
vertex $v_i$ in $G'$.}
\label{fig:embedding}
\end{figure}

In order to estimate the expected number of second category crossings, 
consider the drawing $D$ near a vertex $v_i$ of $G_0$.
In the neighborhood of vertex $v_i$ we 
have at most $d_i$ original edges. Since we draw at most one potential
edge
along 
each original edge,
there can be at most $d_i$ potential edges in the neighborhood. 
Each potential edge can cross each original edge at most once, and any two
potential edges can
cross at most twice. See figure~\ref{fig:embedding}(b). 
Therefore, the total number of first category crossings in $D$ 
in the neighborhood of $v_i$ is at most $2d_i^2$.
(This bound can be substantially improved with a more careful argument, 
see e. g. \cite{FT08}, but we do not need anything better here.)
To obtain the drawing of $G$, we keep each of the potential edges with
probability
$1/p$,
so the expected number of crossings in the neighborhood of $v_i$ is 
at most $({1\over p}+{1\over p^2})d_i^2$, using the fact that the self-crossings 
of the potential $uv$-edges have been eliminated.

Therefore, the total expected number of crossings in the random drawing
of $G$ is at most
$(1+{2\over p}+{1\over p^2})\cro(G')+({1\over p}+{1\over p^2})\sum_{i=1}^{\ell}d_i^2$.

There exists an embedding with at most this many crossings, therefore,
by Lemma~\ref{lemma_FT} we have

\begin{align*}
\cro(G) &\le \left(1+{1\over p}\right)^2 \cro(G') 
+ \left({1\over p}+{1\over p^2}\right) \sum_{i=1}^{\ell}d_i^2 \\
&\le\left(1+{1\over p}\right)^2 \cro(G') + \left({65\over p}+{65\over p^2}\right)\cro(G).
\end{align*}

It follows that

$$\left(1-{65\over p}-{65\over p^2}\right)\cro(G)\le \left(1+{1\over p}\right)^2 \cro(G'),$$

so

\begin{align*}
\left(1-{65\over p}-{65\over p^2}\right)\left(1-{1\over p}\right)^2\cro(G) &\le \left(1-{1\over p^2}\right)^2 \cro(G'), \\
\left(1-{65\over p}-{65\over p^2}\right)\left(1-{2\over p}\right)\cro(G) &\le \cro(G'), \\
\left(1-{67\over p}\right)\cro(G) &\le \cro(G'),
\end{align*}

consequently, 
$$\left(1-\gamma\right)\cro(G) \le \cro(G').$$
\qed


\section{Concluding remarks}

In the statement of our Theorem we cannot 
require that {\em every} subgraph $G'$ with 
$(1-\delta)m(G)$ edges has crossing number
$\cro(G')\ge (1-\gamma)\cro(G)$, instead of just {\em one} such subgraph
$G'$.
In fact, the following statement holds.

\begin{proposition}\label{prop_1}
For every $\varepsilon \in (0,1)$ there exist graphs $G_n$ with
$n(G_n)=\Theta(n)$
vertices and $m(G_n)=\Theta(n^{1+\varepsilon})$ edges with subgraphs
$G'_n\subset G_n$
such that
$$m(G'_n)=\left(1-o(1)\right)m(G_n)$$ 
and 
$$\cro(G'_n)=o(\cro(G_n)).$$ 
\end{proposition}

\begin{proof}
Roughly speaking, $G_n$ will be the disjoint union of a large graph $G'_n$
with low
crossing number and a small graph $H_n$
with large crossing number. More precisely, 
let $G=G_n$ be a disjoint union of graphs $G'=G'_n$ and $H=H_n$, 
where $G'$ is a disjoint union
of $\Theta(n^{1-\varepsilon})$ complete graphs, each with
$\lfloor n^{\varepsilon} \rfloor$
vertices and $H$ is a complete graph with $\lfloor n^{(3+5\varepsilon)/8} \rfloor$
vertices. We have $m(G)=\Theta(n^{1+\varepsilon})$ and
$m(H)=\Theta(n^{(3+5\varepsilon)/4})=o(m(G))$, since
$\frac{3+5\varepsilon}{4}
<1+\varepsilon$. By the crossing lemma (see e. g. \cite{PRTT06}), 
$\cro(G)\ge \cro(H) =
\Omega(n^{(3+5\varepsilon)/2})$, but $\cro(G')= O(n^{1-\varepsilon} \cdot
n^{4\varepsilon})=O(n^{1+3\varepsilon})=o(\cro(G))$, because
$\frac{3+5\varepsilon}{2} > 1+3\varepsilon$.
\end{proof}

In the preliminary version of this paper~\cite{CKT08} we conjectured that
we can require that a positive fraction of all subgraphs 
$G'$ of $G$ with 
$(1-\delta)m(G)$ edges has crossing number
$\cro(G')\ge (1-\gamma)\cro(G)$. The following construction shows that the conjecture does not hold in general for graphs with less than $n^{4/3-\Omega(1)}$ edges.

\begin{proposition}\label{prop_4/3}
For every $\varepsilon \in (0,1/3)$ and $\delta > 0$ there exist graphs $G_n$ with
$n(G_n)=\Theta(n)$
vertices and $m(G_n)=\Theta(n^{1+\varepsilon})$ edges with the following property.
Let $G'_n$ be a random subgraph 
of $G_n$ such that we choose each edge of $G_n$ independently with probability
$p=1-\delta$. Then 

$${\mbox {\rm Pr}}\left[\cro(G'_n)\le o(\cro(G_n))\right]>1-e^{-\delta n^{\Omega(1/3-\varepsilon)}}.$$
\end{proposition}

\begin{proof}
As in Proposition~\ref{prop_1}, the idea is to build the graph $G=G_n$ from two disjoint graphs $K$ and $H$, where $K$ is a large graph with low crossing number and $H$ is a small graph with high crossing number. 
In addition, deleting a random constant fraction of edges from $H$ will break all the crossings in $H$ with high probability.

Now we describe the constructions more precisely. Let $\gamma>0$ be a constant such that $3\varepsilon+4\gamma<1$ and $3\varepsilon+5\gamma>1$. Let $K$ be a disjoint union of $\Theta(n^{1-\varepsilon})$ complete graphs, each with
$n^{\varepsilon}$ vertices (we omit the explicit rounding to keep the notation simple). 
We have $m(K)=\Theta(n^{1+\varepsilon})$ and $\cro(K)=\Theta(n^{1+3\varepsilon})$.

The graph $H$ consists of five {\em main\/} vertices $v_1, v_2, \dots, v_5$ and $n^{1-2\gamma}$ internally vertex disjoint paths of length $n^{\gamma}$ connecting each pair $v_i, v_j$. The graph $H$ has $n(H)=\Theta(n^{1-\gamma})$ vertices and $m(H)=\Theta(n^{1-\gamma})$ edges. We claim that $\cro(H)=n^{2-4\gamma}$. The upper bound follows from the fact that the crossing number of $K_5$ is $1$. We take a drawing of $K_5$ with one crossing and replace each edge $e$ by $n^{1-2\gamma}$ paths drawn close to $e$. For the lower bound, take a drawing of $H$ minimizing the number of crossings. Let $p_{i,j}$ be a path with the minimum number of crossings among the paths connecting $v_i$ and $v_j$. By redrawing all the other paths connecting $v_i$ and $v_j$ along $p_{i,j}$ the crossing number of the drawing does not change. The paths $p_{i,j}$ together form a subdivision of $K_5$, therefore at least one pair $p_{i,j}, p_{k,l}$ of the paths crosses. Due to the redrawing, every path connecting $v_i$ and $v_j$ crosses every path connecting $v_k$ and $v_l$, which makes $n^{2-4\gamma}$ crossings. By the choice of $\gamma$, we have $n^{1+3\varepsilon} = o(n^{2-4\gamma})$, therefore $\cro(G)=\Theta(\cro(H))$ and $\cro(K)=o(\cro(G))$. 

Let $G'$ be a random subgraph of $G$ where each edge of $G$ is taken independently with probability
$p=1-\delta$. Let $H'=G'\cap H$. We show that with high probability, $H'$ is a forest, in particular $\cro(H')=0$. This happens if at least one edge is missing from every path connecting two main vertices of $H$. The probability of such an event is at least 
$$1-n\cdot(1-\delta)^{n^\gamma} \ge 1-e^{-\delta n^{\gamma}+\log{n}} \ge 1-e^{-\delta n^{\Omega(1/3-\varepsilon)}}.$$
It follows that with this probability, $\cro(G') \le \cro(K) \le o(\cro(G))$.
\end{proof}

Note that in the previous construction the number $\delta$ does not have to be constant: it is enough to delete a random $\delta=c\log{n}/n^{\gamma}$ fraction of the edges to get the same conclusion with probability almost $1$.

The question whether deleting a small random constant fraction of the edges of a graph $G$ decreases the crossing number only by a small constant fraction remains open for graphs with more than $n^{4/3}$ edges. We do not know the answer even to the following weaker version of the question.

\begin{problem}
Let $\varepsilon \in (0,2/3)$ and $p\in (0,1)$ be constants. Does there exist $c(p)>0$ and $n_0$ such that for every graph $G$ with $n(G)>n_0$ and $m(G) > n(G)^{4/3+\varepsilon}$, a random subgraph $G'$ of $G$ with each edge taken with probability $p$ has crossing number at least $c(p)\cdot \cro(G)$, with probability at least $1/2$? 
\end{problem}

The graphs in Proposition~\ref{prop_4/3} have small number of edges responsible for almost all the crossings. Is this the only way how to force a random subgraph of $G$ to have crossing number $o(\cro(G))$?  

\begin{problem}
Let $\varepsilon > 0$. Does there exist $n_0$ and $\delta$ 
such that every graph $G$ with $n(G)\ge n_0$ and $m(G)\ge n(G)^{1+\varepsilon}$ has a subset $F$ of $o(m(G))$ edges such that 
every subgraph $G'$ of $G$ with $m(G')\ge (1-\delta)m(G)$ and $E(G')\subset E(G)\setminus F$ has $\cro(G')\ge
(1-\varepsilon)\cro(G)$?
\end{problem}

\end{document}